\documentclass{amsart}
\usepackage{amsmath}
\usepackage{amsthm}
\usepackage{amssymb}
\usepackage{amsfonts,mathrsfs}
\usepackage{stmaryrd}
\usepackage{amsxtra}  
\usepackage{epsfig}
\usepackage{verbatim}
\usepackage{color}

\theoremstyle{plain}
\newtheorem{theorem}[equation]{Theorem}
\newtheorem{proposition}[equation]{Proposition}
\newtheorem{lemma}[equation]{Lemma}
\newtheorem{corollary}[equation]{Corollary}

\theoremstyle{remark}

\numberwithin{equation}{section}

\newcommand{\ssubset}{\subset\subset}

\def\Re{\operatorname{Re}}

\def\supp{\operatorname{supp}}
\def\I{\operatorname{I}}
\def\II{\operatorname{II}}
\def\III{\operatorname{III}}
\def\IV{\operatorname{IV}}
\def\V{\operatorname{V}}
\def\Span{\operatorname{span}}

\begin{document}

\bibliographystyle{plain}
\title[Plurisubharmonic defining functions]{A note on plurisubharmonic defining functions in $\mathbb{C}^{n}$}
\author{J. E. Forn\ae ss, A.-K. Herbig}
\thanks{Research of the first author was partially supported by an NSF grant.}
\thanks{Research of the second author was supported by FWF grant P19147}
\address{Department of Mathematics, \newline University of Michigan, Ann Arbor, Michigan 48109, USA}
\email{fornaess@umich.edu}
\address{Department of Mathematics, \newline University of Vienna, Vienna,  A-1090, Austria}
\email{anne-katrin.herbig@univie.ac.at}
\date{}
\begin{abstract}
  Let $\Omega\ssubset\mathbb{C}^{n}$, $n\geq 3$, be a smoothly bounded domain. Suppose that 
  $\Omega$ admits 
  a  smooth defining function which is plurisubharmonic on the boundary of $\Omega$. Then the 
  Diederich--Forn\ae ss exponent can be chosen arbitrarily close to $1$, and the closure of $\Omega$ 
  admits a Stein neighborhood basis.
\end{abstract}
\maketitle

\section{Introduction}
   Let $\Omega\ssubset\mathbb{C}^{n}$ be a smoothly bounded domain. Throughout, we  suppose
   that $\Omega$ admits a $\mathcal{C}^{\infty}$-smooth defining function $\rho$ which is 
   plurisubharmonic on the boundary, $b\Omega$, of $\Omega$. That is,
    \begin{align}\label{E:psh}
      H_{\rho}(\xi,\xi)(z):=\sum_{j,k=1}^{n}\frac{\partial^{2}\rho}{\partial z_{j}\partial\overline{z}_{k}}(z)
      \xi_{j}\overline{\xi}_{k}\geq 0\;\;\text{for all}\;z\in b\Omega,\;\xi\in\mathbb{C}^{n}.
    \end{align}
   The question we are concerned with is what condition \eqref{E:psh} tells us about the
   behaviour of the complex Hessian of $\rho$ - or of some other defining function of $\Omega$ - 
   away from the boundary of $\Omega$.
   
   That $\rho$ is not necessarily plurisubharmonic in any neighborhood of $b\Omega$ can be seen
   easily, for an example see Section 2.3 in \cite{For-Her}. In \cite{For-Her}, we showed that if $n=2$,
   then for any $\epsilon>0$, $K>0$ there exist smooth defining functions
   $\rho_{i}$, $i=1,2$, and a neighborhood $U$ of $b\Omega$ such that
   \begin{align}\label{E:dim2result}
     H_{\rho_{i}}(\xi,\xi)(q_{i})\geq-\epsilon|\rho_{i}(q_{i})|\cdot|\xi|^{2}
     +K|\langle\partial\rho_{i}(q_{i}),\xi\rangle|^{2}
   \end{align}
   for all $\xi\in\mathbb{C}^{2}$ and $q_{1}\in\overline{\Omega}\cap U$, 
   $q_{2}\in\Omega^{c}\cap U$.
   
   The estimates \eqref{E:dim2result} imply the existence of particular exhaustion functions  for 
   $\Omega$  and the complement of $\overline{\Omega}$, 
   which is not a direct consequence of \eqref{E:psh}. 
   A Diederich--Forn\ae ss 
   exponent of a domain 
   is a number $\tau\in(0,1]$ for which there exists a smooth defining function $s$ such
   that $-(-s)^{\tau}$ is strictly plurisubharmonic in the domain. 
   It was shown in \cite{DF1,Ran} that all smoothly 
   bounded, pseudoconvex domains have a Diederich--Forn\ae ss exponent. However, it is also known 
   that there are smoothly bounded, pseudoconvex domains for which the largest 
   Diederich--Forn\ae ss exponent has to be chosen arbitrarily close to $0$ (see \cite{DF2}). In 
   \cite{For-Her}, we showed that \eqref{E:dim2result}, $i=1$, implies that the Diederich--Forn\ae ss
   exponent can be chosen arbitrarily close to $1$. We also showed that 
   \eqref{E:dim2result}, $i=2$, yields that the complement of $\Omega$ can be exhausted by bounded, 
   strictly plurisubharmonic functions. In particular, the closure of $\Omega$ admits a Stein
   neighborhood basis.
   
   \medskip
   
   For $n\geq 3$ we obtain the following:
   \begin{theorem}\label{T:Main}
  Let $\Omega\ssubset\mathbb{C}^{n}$ be a smoothly bounded domain. Suppose that 
  $\Omega$ has a smooth defining function which is plurisubharmonic on the boundary of 
  $\Omega$. Then for any $\epsilon>0$ there exist a neighborhood $U$ of $b\Omega$ and smooth  
  defining functions $r_{1}$ and $r_{2}$ such that
  \begin{align}\label{E:Main1}
    H_{r_{1}}(\xi,\xi)(q)\geq -\epsilon
    \left[|r_{1}(q)|\cdot|\xi|^{2}+\frac{1}{|r_{1}(q)|}\cdot\left|\langle\partial r_{1}(q),\xi\rangle\right|^{2}\right]
  \end{align}
  holds for all $q\in\Omega\cap U$, $\xi\in\mathbb{C}^{n}$, and
  \begin{align}\label{E:Main2}
    H_{r_{2}}(\xi,\xi)(q)\geq 
    -\epsilon
    \left[r_{2}(q)\cdot|\xi|^{2}+\frac{1}{r_{2}(q)}\cdot\left|\langle\partial r_{2}(q),\xi\rangle\right|^{2}\right]  
  \end{align}
  holds for all $q\in(\overline{\Omega})^{c}\cap U$, $\xi\in\mathbb{C}^{n}$.
\end{theorem}
Let us remark that our proof of Theorem \ref{T:Main} also works when $n=2$. However,
the results of Theorem \ref{T:Main} are weaker than \eqref{E:dim2result}. Nevertheless, 
they are still strong enough to obtain that the Diederich--Forn\ae ss exponent can be chosen arbitrarily close to 1 and that the closure of the domain admits a Stein neighborhood basis. In particular, we have the following: 

\begin{corollary}\label{C:DF}
  Assume the hypotheses of Theorem \ref{T:Main} hold. Then
  \begin{enumerate}
    \item for all $\eta\in (0,1)$ there exists a smooth defining function $\tilde{r}_{1}$ of $\Omega$ such that
      $-(-\tilde{r}_{1})^{\eta}$ is strictly plurisubharmonic on $\Omega$,
    \item for all $\eta>1$ there exist a smooth defining function $\tilde{r}_{2}$ of $\Omega$ and a  
      neighborhood $U$ of $\overline{\Omega}$ such that $\tilde{r}_{2}^{\eta}$ is strictly plurisubharmonic
      on $(\overline{\Omega})^{c}\cap U$.
  \end{enumerate}
\end{corollary}

We note that in \cite{DF3} it was proved that (i) and (ii) of Corollary \ref{C:DF}  hold for 
so-called regular 
domains. Furthermore, in \cite{DF4} it was shown that pseudoconvex domains with real-analytic boundary are regular domains.

This article is structured as follows. In Section \ref{S:prelim}, we give the setting and define our basic notions. Furthermore, we show in this section which piece of the complex Hessian of $\rho$ at a given point $p$ in $b\Omega$ constitutes an obstruction for inequality \eqref{E:Main1} to hold for a given $\epsilon>0$. In Section \ref{S:modification}, we construct a local defining function which does not possess this obstruction term  to \eqref{E:Main1} at a given boundary point $p$. Since this fixes our problem with \eqref{E:Main1} only at this point $p$ (and at nearby boundary points at which the Levi form is of the same rank as at $p$), we will need to patch the newly constructed local defining functions without letting the obstruction term arise again. This is done
in Section \ref{S:cutoff}. In Section \ref{S:proof}, we finally prove \eqref{E:Main1} and remark at the end how to obtain \eqref{E:Main2}. We conclude this paper with the proof of Corollary
\ref{C:DF} in Section \ref{S:DF}.

We would like to thank J. D. McNeal for fruitful discussions on this project, in particular we are very grateful to him for providing us with Lemma \ref{L:McNeal} and its proof.

\medskip

\section{Preliminaries and pointwise  obstruction}\label{S:prelim}

Let $(z_{1},\dots,z_{n})$ denote the coordinates of $\mathbb{C}^{n}$. We shall identify the vector
$\langle\xi_{1},\dots,\xi_{n}\rangle$ in $\mathbb{C}^{n}$ with 
$\sum_{i=1}^{n}\xi_{i}\frac{\partial}{\partial z_{i}}$ in the $(1,0)$-tangent bundle of $\mathbb{C}^{n}$ at
any given point. This means in particular that if $X$, $Y$ are $(1,0)$-vector fields with
$X(z)=\sum_{i=1}^{n}X_{i}(z)\frac{\partial}{\partial z_{i}}$ and 
$Y(z)=\sum_{i=1}^{n}Y_{i}(z)\frac{\partial}{\partial z_{i}}$, then
\begin{align*}
  H_{\rho}(X,Y)(z)=\sum_{j,k=1}^{n}\frac{\partial^{2}\rho}{\partial z_{j}\partial\overline{z}_{k}}(z)
  X_{j}(z)\overline{Y}_{k}(z).
\end{align*}

Suppose $Z$ is another $(1,0)$-vector field with 
$Z(z)=\sum_{l=1}^{n}Z_{l}(z)\frac{\partial}{\partial z_{l}}$. For notational convenience, and because of lack of a better notation, we shall write
\begin{align*}
 (ZH_{\rho})(X,Y)(z):=
  \sum_{j,k,l=1}^{n}\frac{\partial^{3}\rho}{\partial z_{j}\partial\overline{z}_{k}\partial z_{l}}(z)
  X_{j}(z)\overline{Y}_{k}(z) Z_{l}(z).
\end{align*}

We use the pointwise hermitian inner product $\langle .,.\rangle$ defined by
$\langle\frac{\partial}{\partial z_{j}},\frac{\partial}{\partial z_{k}}\rangle=\delta_{j}^{k}$. Hoping that it will not cause any confusion, we also write $\langle .,.\rangle$ for contractions of vector fields and forms.

We will employ the so-called (sc)-(lc) inequality: $|ab|\leq\tau|a|^{2}+\frac{1}{4\tau}|b|^{2}$ for 
$\tau>0$. Furthermore, we shall write $|A|\lesssim|B|$ to mean $|A|\leq c|B|$ for some constant $c>0$ which does not depend on any of the relevant parameters. In particular, we will only use this notation when $c$ depends solely on absolute constants, e.g., dimension, quantities related to the given defining function 
$\rho$.

\medskip

 Let us now work on proving inequality \eqref{E:Main1}.
 Since $b\Omega$ is smooth, there exists a neighborhood $U$ of $b\Omega$ and a smooth map
  \begin{align*}
    \pi:\overline{\Omega}\cap U&\longrightarrow b\Omega\\
    q&\longmapsto\pi(q)=p
  \end{align*}  
     such that $\pi(q)=p$ lies on the line normal to $b\Omega$ passing through $q$ and $|p-q|$ equals 
     the  Euclidean distance, $d_{b\Omega}(q)$, of $q$ to $b\Omega$.
     
After possibly shrinking $U$, we can assume that $\partial\rho\neq 0$ on $U$. We set 
$N(z)=\frac{1}{|\partial\rho(z)|}\sum_{j=1}^{n}\frac{\partial\rho}{\partial\overline{z}_{j}}(z)
 \frac{\partial}{\partial z_{j}}$. If $f$ is a smooth function on $U$, then it follows from Taylor's theorem 
 that 
 \begin{align}\label{E:BasicTaylor}
     f(q)=f(p)-d_{b\Omega}(q)\left(\Re N\right)(f)(p)+\mathcal{O}\left(d_{b\Omega}^{2}(q)\right)\;\;
     \text{for}\;\;q\in \overline{\Omega}\cap U;
 \end{align}
 for details see for instance Section 2.1 in \cite{For-Her}\footnote{Equation \eqref{E:BasicTaylor} above differs from (2.1) in \cite{For-Her} by a factor of $2$ in the second term on the right hand side.  This stems from mistakenly using that outward normal of length $1/2$ instead of the one of unit length in 
 \cite{For-Her}. However, this mistake is inconsequential for the results in \cite{For-Her}.}.
 Let $p\in b\Omega\cap U$ be given. Let $W\in\mathbb{C}^{n}$ be a weak, complex tangential vector at 
 $p$, i.e., $\langle\partial\rho(p),W\rangle=0$ and $H_{\rho}(W,W)(p)=0$. If $q\in\Omega\cap U$ with 
 $\pi(q)=p$, then \eqref{E:BasicTaylor} implies
  \begin{align}\label{E:BasicTaylorH}
   H_{\rho}(W,W)(q)=H_{\rho}(W,W)(p)-d_{b\Omega}(q)\left(\Re N\right)
   \left(H_{\rho}(W,W)\right)(p)+\mathcal{O}(d^{2}_{b\Omega}(q))|W|^{2}.
 \end{align}
 Since $H_{\rho}(W,W)$ is a real-valued function, we have
 \begin{align*}
   \left(\Re N\right)\left(H_{\rho}(W,W)\right)=\Re\left[N\left(H_{\rho}(W,W)\right)\right].
 \end{align*}
 Moreover,
 $H_{\rho}(W,W)$ is non-negative on $b\Omega\cap U$ and equals $0$ at $p$. That is,
 $H_{\rho}(W,W)_{|_{b\Omega\cap U}}$ attains a local minimum at $p$. Therefore, any tangential 
 derivative of $H_{\rho}(W,W)$ vanishes at $p$. Since $N-\overline{N}$ is tangential to $b\Omega$, 
 we obtain 
 \begin{align*}
   \Re\left[N\left(H_{\rho}(W,W)\right)\right](p)=N\left(H_{\rho}(W,W)\right)(p)
   =(NH_{\rho})(W,W)(p),
 \end{align*}  
 where the last equality holds since $W$ is a fixed vector.
 Hence, \eqref{E:BasicTaylorH} becomes
 \begin{align}\label{E:BasicTaylorHW}
   H_{\rho}(W,W)(q)=-d_{b\Omega}(q)(NH_{\rho})(W,W)(p)
   +\mathcal{O}\left(d_{b\Omega}^{2}(q)\right)|W|^{2}.
 \end{align}
 Clearly,  we have a problem with obtaining \eqref{E:Main1} 
 when $(NH_{\rho})(W,W)$ is strictly positive at $p$.  That is, when $H_{\rho}(W,W)$ 
 is strictly decreasing  along the real inward normal to $b\Omega$ at $p$, i.e.,
 $H_{\rho}(W,W)$ becomes negative there,  then \eqref{E:Main1} can not hold for  the complex
 Hessian $\rho$ when $
 \epsilon>0$ is sufficiently close to zero. The question is whether we can find another smooth defining function $r$ of 
 $\Omega$ such that $(NH_{r})(W,W)(p)$ is less than $(NH_{\rho})(W,W)(p)$. The construction of
 such a function $r$ is relatively easy and straightforward when $n=2$ (see Section 2.3 in \cite{For-Her} 
 for a non-technical derivation of $r$). The difficulty in higher dimensions arises simply from the fact that  
 the Levi form of a defining function might vanish in more than one 
 complex tangential direction at a given boundary point.
 
 \medskip
 
 \section{Pointwise Modification of $\rho$}\label{S:modification}

 Let $\Sigma_{i}\subset b\Omega$ be the set of boundary points at which the Levi form of $\rho$
 has rank $i$, $i\in\{0,\dots,n-1\}$. Note that $\cup_{i=0}^{j}\Sigma_{i}$ is closed in $b\Omega$ for any 
 $j\in\{0,\dots,n-1\}$. Moreover, 
 $\Sigma_{j}$ is relatively closed in $b\Omega\setminus \cup_{i=0} ^{j-1}\Sigma_{i}$  for 
 $j\in\{1,\dots,n-1\}$. Of course, $\Sigma_{n-1}$ is the set of strictly pseudoconvex
  boundary points of $\Omega$.

 Let $p\in b\Omega\cap\Sigma_{i}$ for some $i\in\{0,\dots, n-2\}$ be given. Then there exist a 
 neighborhood $V\subset U$ of $p$ and smooth, linearly independent $(1,0)$-vector fields $W^{\alpha}$, 
 $1\leq\alpha\leq n-1-i$, on $V$, which are complex tangential to $b\Omega$ on $b\Omega\cap V$ and 
 satisfy $H_{\rho}(W^{\alpha},W^{\alpha})=0$ on $\Sigma_{i}\cap V$.
 We consider those points $q\in\Omega\cap V$ with $\pi(q)=p$.  
 
 \medskip
 
 We shall work with the smooth  function
 \begin{align*}
    r(z)=\rho(z)\cdot e^{-C\sigma(z)},\;\text{where}\; \;
    \sigma(z)=\sum_{\alpha=1}^{n-1-i}H_{\rho}(W^{\alpha},W^{\alpha})(z)
 \end{align*}
 for $z\in V$. Here, the constant $C>0$ is fixed and to be chosen later.
 Note that $r$ defines $b\Omega$ on $b\Omega\cap V$.
 Furthermore, $\sigma$ is a smooth function on $V$ which is non-negative on $b\Omega\cap V$ and 
 vanishes on the set $\Sigma_{i}\cap 
 V$. That means that $\sigma_{|_{b\Omega\cap V}}$ attains a local minimum at each point in $\Sigma_{i}
 \cap V$. Therefore, any tangential derivative of $\sigma$ vanishes on $\Sigma_{i}\cap V$. 
 Moreover, if $z\in\Sigma_{i}\cap V$ and  $T\in\mathbb{C}T_{z}b\Omega$ is such that $H_{\rho}(T,T)$ 
 vanishes at $z$, then $H_{\sigma}(T,T)$ is non-negative at that point. 
   
  \medskip 
   
 Let $W\in 
 \mathbb{C}^{n}$
 be a  vector contained in the span of the vectors 
 $\left\{W^{\alpha}(p)\right\}_{\alpha=1}^{n-1-i}$. Then, using \eqref{E:BasicTaylorHW}, it follows that
 \begin{align}\label{E:explaincutoff}
   H_{r}(W,W)(q)
   =&e^{-C\sigma(q)}\left[H_{\rho}(W,W)
   -C\Re\left(\langle\partial\rho,W\rangle\overline{\langle\partial\sigma,W\rangle
   }\right)\right.\notag\\
   &\hspace{3.5cm}+\left.\rho\left(C^{2}\left|\langle\partial\sigma,W\rangle\right|^{2} 
   -CH_{\sigma}(W,W)\right)\right](q)\notag\\
   =&e^{-C\sigma(q)}
   \Bigl[
   -d_{b\Omega}(q)(NH_{\rho})(W,W)(p)-C\rho(q)H_{\sigma}(W,W)(q)\bigr. 
   +\mathcal{O}\left(d_{b\Omega}^{2}(q)\right)|W|^{2}\notag\\
   &\hspace{1.5cm}+C^{2}\rho(q)\left|\langle\partial\sigma(q),W\rangle \right|^{2}\bigl.-2C\Re\left(
   \langle\partial\rho,W\rangle\overline{\langle\partial\sigma,W\rangle
   }
   \right)(q)
   \Bigr].
 \end{align}
 Since $\langle\partial\sigma(p),W\rangle=0=\langle\partial\rho(p),W\rangle$, Taylor's theorem gives
 \begin{align*}
   \langle\partial\sigma(q),W\rangle=\mathcal{O}\left(r(q)\right)|W|=\langle\partial\rho(q),W\rangle.
 \end{align*}
 Therefore, we obtain
 \begin{align}\label{E:BasicTaylorr}
   H_{r}(W,W)(q)\geq -d_{b\Omega}(q)e^{-C\sigma(q)}(NH_{\rho})(W,W)(p)&-Cr(q)H_{\sigma}(W,W)(q)\\
   &+\mathcal{O}\left(r^{2}(q)\right)|W|^{2},\notag
\end{align}
 where the constant in the last term depends on the choice of the constant $C$. However, in view of our claim \eqref{E:Main1}, this  is inconsequential. From here on, we will not point out such negligible dependencies. 
 
 We already know that $H_{\sigma}(W,W)(p)$ is non-negative, i.e., of the right sign to correct 
 $(NH_{\rho})(W,W)(p)$ when necessary. The question is whether the sizes of 
 $(NH_{\rho})(W,W)(p)$ and 
 $H_{\sigma}(W,W)(p)$ are comparable in some sense. The following proposition clarifies this.
 \begin{proposition}\label{P:compare}
   There exists a constant $K>0$ such that
   \begin{align*}
     \left|(NH_{\rho})(W,W)(z_{0})\right|^{2}
     \leq K|W|^{2}\cdot H_{\sigma}(W,W)(z_{0})
   \end{align*}
   holds for all $z_{0}\in\Sigma_{i}\cap V$ and $W\in\mathbb{C}T_{z_{0}}b\Omega$
   with $H_{\rho}(W,W)(z_{0})=0$.
 \end{proposition}
 In order to prove Proposition \ref{P:compare}, we need the following lemma:
 \begin{lemma}\label{L:compare}
   Let $z_{0}\in b\Omega$ and $U$ a neighborhood of $z_{0}$. Let $Z$ be a smooth $(1,0)$-vector field  
   defined 
   on $U$, which is complex tangential to $b\Omega$ on $b\Omega\cap U$, and let 
   $Y\in\mathbb{C}^{n}$ be a vector belonging to $
   \mathbb{C}T_{z_{0}}b\Omega$. Suppose that $Y$ and $Z$ are such that
   \begin{align*}
     H_{\rho}(Y,Y)(z_{0})=0=H_{\rho}(Z,Z)(z_{0}).
   \end{align*}
   Set $X=\sum_{j=1}^{n}\overline{Y}(Z_{j})\frac{\partial}{\partial z_{j}}$. Then the following holds:
   \begin{enumerate}
     \item $X$ is complex tangential to $b\Omega$ at $z_{0}$,
     \item $\left(YH_{\rho}\right)(X,Z)(z_{0})=0$,\
     \item $H_{H_{\rho}(Z,Z)}(Y,Y)(z_{0})\geq H_{\rho}(X,X)(z_{0})$.
   \end{enumerate}
 \end{lemma}
 \begin{proof}[Proof of Lemma \ref{L:compare}]
   (1) That $X$ is complex tangential to $b\Omega$ at $z_{0}$ was shown in Lemma 3.4 of 
   \cite{For-Her}.
   
   (2) The plurisubharmonicity of $\rho$ says that 
   both $H_{\rho}(Y,Y)_{|_{b\Omega\cap U}}$ and $H_{\rho}(Z,Z)_{|_{b\Omega\cap U}}$
   attain a local minimum at $z_{0}$. In fact, the function $H_{\rho}(aY+bZ,aY+bZ)_{|_{b\Omega}}$, 
   $a,b\in\mathbb{C}$, attains a local minimum at $z_{0}$. This means that any tangential derivative of
   either one of those three functions must vanish at that point. In particular, we have 
   \begin{align*}
    0&=\langle\partial H_{\rho}(aY+bZ,aY+bZ),X\rangle(z_{0})\\
      &=|a|^{2}\langle\partial H_{\rho}(Y,Y),X\rangle(z_{0})
      +2\Re\left(a\overline{b}\langle\partial H_{\rho}(Y,Z),X\rangle
      \right)(z_{0})+|b|^{2}\langle\partial H_{\rho}(Z,Z),X\rangle(z_{0})\\
      &=2\Re\left(a\overline{b}\langle\partial H_{\rho}(Y,Z),X\rangle
      \right)(z_{0}).
   \end{align*}
   Since this is true for all $a,b\in\mathbb{C}$, it follows that $\langle\partial H_{\rho}(Y,Z),X\rangle$ 
   must vanish
   at $z_{0}$. But the plurisubharmonicity of $\rho$ at $z_{0}$ yields
   \begin{align*}
     \langle\partial H_{\rho}(Y,Z),X\rangle=\left(XH_{\rho}\right)(Y,Z)(z_{0})=
     \left(YH_{\rho}\right)(X,Z)(z_{0}),
   \end{align*}   
   which proves the claim.
   
   (3) Consider the function
   \begin{align*}
     f(z)=\left(
     H_{\rho}(Z,Z)\cdot H_{\rho}(X,X)-\left|H_{\rho}(X,Z)\right|^{2}\right)(z)\;\;\text{for}\;\;z\in U.
   \end{align*}
   Note that $f_{|_{b\Omega\cap U}}$ attains a local minimum at $z_{0}$. Since $Y$ is a weak direction at 
   $z_{0}$, it follows that $H_{f}(Y,Y)(z_{0})$ is 
   non-negative. This implies that
   \begin{align}\label{E:comparetemp}
     \left(H_{H_{\rho}(Z,Z)}(Y,Y)\cdot H_{\rho}(X,X)\right)(z_{0})\geq
     \left|
     \langle \partial H_{\rho}(X,Z), Y\rangle(z_{0})
     \right|^{2},
   \end{align}
   where we used that both $H_{\rho}(Z,Z)$ and any tangential derivative of $H_{\rho}(Z,Z)$ at $z_{0}$ 
   are zero. We compute
   \begin{align*}
     \langle\partial H_{\rho}(X,Z),Y\rangle(z_{0})
     =(YH_{\rho})(X,Z)(z_{0})+H_{\rho}(X,X)(z_{0})
     +H_{\rho}\left(\sum_{j=1}^{n}Y(X_{j})\frac{\partial}{\partial z_{j}},Z \right)(z_{0}).
   \end{align*}
   The first term on the right hand side equals zero by part (2) of Lemma \ref{L:compare}, and the third 
   term is zero as well since $\rho$ plurisubharmonic at $z_{0}$ and $Z$ is a weak direction  there. Therefore,
   \eqref{E:comparetemp} becomes
   \begin{align*}
     H_{\rho}(X,X)(z_{0})\leq H_{H_{\rho}(Z,Z)}(Y,Y)(z_{0}).
   \end{align*}
 \end{proof}
 Now we can proceed to show Proposition \ref{P:compare}.
 \begin{proof}[Proof of Proposition \ref{P:compare}]
   Recall that we are working with vectors $W\in\mathbb{C}^{n}$ contained in the span of
   $W^{1}(z_{0}),\dots,W^{n-1-i}(z_{0})$. We consider the function
   \begin{align*}
     h(z)=\left(
       \sigma\cdot H_{\rho}(N,N)-\sum_{\alpha=1}^{n-1-i}\left|
       H_{\rho}(N,W^{\alpha})
       \right|^{2}
     \right)(z) \;\;\text{for}\;\;z\in U,
   \end{align*}
   where $\sigma=\sum_{\alpha=1}^{n-1-i}H_{\rho}(W^{\alpha},W^{\alpha})$. Again, since $\rho$ is 
   plurisubharmonic on $b\Omega$, $h_{|_{b\Omega\cap V}}$ has a local minimum at $z_{0}$. 
   This, together with $H_{\rho}(W,W)(z_{0})=0$, implies that
   $H_{h}(W,W)(z_{0})$ is non-negative. Since both $\sigma$ and $\langle\partial\sigma,
   W\rangle=0$ vanish at $z_{0}$, 
   it 
   follows that
   \begin{align*}
     &\left(
     H_{\sigma}(W,W)\cdot H_{\rho}(N,N)
     \right)(z_{0})\\
     &\geq
     \sum_{\alpha=1}^{n-1-i}\left|
     (WH_{\rho})(N,W^{\alpha})+H_{\rho}\left(\sum_{j=1}^{n}W(N_{j})\frac{\partial}{\partial z_{j}},W^{\alpha}
     \right)
     +H_{\rho}\left(
     N,\sum_{j=1}^{n}\overline{W}(W_{j}^{\alpha})\frac{\partial}{\partial z_{j}}
     \right)
     \right|^{2}(z_{0})\\
      &=
     \sum_{\alpha=1}^{n-1-i}\left|
     (WH_{\rho})(N,W^{\alpha})+H_{\rho}\left(
     N,\sum_{j=1}^{n}\overline{W}(W_{j}^{\alpha})\frac{\partial}{\partial z_{j}}
     \right)
     \right|^{2}(z_{0}),
   \end{align*}
   where the last step follows from $\rho$ being plurisubharmonic at $z_{0}$ and the 
   $W^{\alpha}$'s being 
   weak directions  there.
   Moreover, we have that $(WH_{\rho})(N,W^{\alpha})$ equals $(NH_{\rho})(W,W^{\alpha})$ at $z_{0}$.  Writing
   $X^{\alpha}=\sum_{j=1}^{n}\overline{W}(W_{j}^{\alpha})\frac{\partial}{\partial z_{j}}$, we obtain
   \begin{align*}
     \left(
       H_{\sigma}(W,W)\cdot H_{\rho}(N,N)
     \right)(z_{0})
     &\geq
     \sum_{\alpha=1}^{n-1-i}\left|
       (NH_{\rho})(W,W^{\alpha})+H_{\rho}(N,X^{\alpha})
     \right|^{2}(z_{0})\\
     &\geq
     \sum_{\alpha=1}^{n-1-i}\left(
     \frac{1}{2}\left|
     (NH_{\rho})(W,W^{\alpha})
     \right|^{2}-3
     \left|H_{\rho}(N,X^{\alpha})\right|^{2}
     \right)(z_{0}).
   \end{align*}
   Here the last step follows from the (sc)-(lc) inequality.
   Since $\rho$ is plurisubharmonic at $z_{0}$, we can apply the Cauchy--Schwarz inequality
   \begin{align*}
     \left|
     H_{\rho}(N,X^{\alpha})(z_{0})
     \right|^{2}
     &\leq
     \left(
     H_{\rho}(N,N)\cdot H_{\rho}(X^{\alpha},X^{\alpha})
     \right)(z_{0})\\
    &\leq \left(
     H_{\rho}(N,N)\cdot H_{H_{\rho}(W^{\alpha},W^{\alpha})}(W,W)
     \right)(z_{0}),
   \end{align*}
   where the last estimate follows by part (3) of Lemma \ref{L:compare} with $W$ and $W^{\alpha}$ 
   in place of $Y$ and $Z$, respectively. Thus we have
   \begin{align*}
    \sum_{\alpha=1}^{n-1-i}\left|H_{\rho}(N,X^{\alpha})(z_{0}) \right|^{2}
    \leq
    \left(H_{\rho}(N,N)\cdot H_{\sigma}(W,W)\right)(z_{0}),
   \end{align*}
   which implies that
   \begin{align}\label{E:comparetemp2}
     \sum_{\alpha=1}^{n-1-i}\left|
     (NH_{\rho})(W,W^{\alpha})(z_{0})
     \right|^{2}
     \leq
     8\left(H_{\sigma}(W,W)\cdot H_{\rho}(N,N)\right)(z_{0}).
   \end{align} 
   Since $W$ is a linear combination of $\{W^{\alpha}(z_{0})\}_{\alpha=1}^{n-1-i}$, we can write
   $W=\sum_{\alpha=1}^{n-1-i}a_{\alpha}W^{\alpha}(z_{0})$ for some scalars $a_{\alpha}\in\mathbb{C}$.
   Because of the linear independence of the $W^{\alpha}$'s on $V$, there exists a constant $K_{1}>0$ 
   such 
   that
   \begin{align*}
     \sum_{\alpha}^{n-1-i}|b_{\alpha}|^{2}\leq K_{1}\left|
     \sum_{\alpha=1}^{n-1-i}b_{\alpha}W^{\alpha}(z)
     \right|^{2}\;\text{for all}\;z\in b\Omega\cap V,\;b_{\alpha}\in\mathbb{C}.
   \end{align*}
   Thus it follows that
   \begin{align*}
     \sum_{\alpha=1}^{n-1-i}
     \left|
     (NH_{\rho})(W,W^{\alpha})(z_{0})
     \right|^{2}
     &\geq
     \frac{1}{K_{1}|W|^{2}}
     \sum_{\alpha=1}^{n-1-i}\left|(NH_{\rho})(W,a_{\alpha}W^{\alpha})(z_{0})\right|^{2}\\
     &\geq
     \frac{1}{K_{1}(n-1-i)|W|^{2}}
     \left|(NH_{\rho})(W,W)(z_{0})\right|^{2}.
   \end{align*}
   Hence, \eqref{E:comparetemp2} becomes
   \begin{align*}
     \left|(NH_{\rho})(W,W)(z_{0})\right|^{2}\leq
     8K_{1}(n-1-i)|W|^{2}\left(H_{\sigma}(W,W)\cdot H_{\rho}(N,N)\right)(z_{0}).
   \end{align*}
   Let $K_{2}>0$ be a constant such that
   $H_{\rho}(N,N)_{|_{b\Omega}}\leq K_{2}$ holds. Setting $K=8K_{1}K_{2}(n-1-i)$, 
   it follows that
   \begin{align*}
     \left|(NH_{\rho})(W,W)(z_{0})\right|^{2}\leq K|W|^{2}H_{\sigma}(W,W)(z_{0}).
   \end{align*}
 \end{proof}
 
 Recall that we are considering a fixed boundary point $p\in\Sigma_{i}$ and all $q\in\Omega\cap V$, 
 $\pi(q)=p$, for 
 some sufficiently small neighborhood of $p$. After possibly shrinking $V$ it follows by Taylor's theorem 
 that
 \begin{align*}
   H_{\sigma}(W,W)(q)=H_{\sigma}(W,W)(\pi(q))+\mathcal{O}\left(d_{b\Omega}(q)\right)|W|^{2}
 \end{align*}
 holds for all $q\in\Omega\cap V$ with $\pi(q)=p$.
 Using this and Proposition \ref{P:compare}, we get for $q\in\Omega\cap V$ with 
 $\pi(q)=p$ that 
 \begin{align*}
     \left|(NH_{\rho})(W,W)(p)\right|^{2}\leq K|W|^{2}\left[H_{\sigma}(W,W)(q)
     +\mathcal{O}\left(d_{b\Omega}(q)\right)|W|^{2}\right].
 \end{align*}
 Therefore, our basic estimate \eqref{E:BasicTaylorr} of the complex Hessian of $r$ in direction $W$
 becomes
 \begin{align*}
   H_{r}(W,W)(q)\geq 
  - d_{b\Omega}(q)e^{-C\sigma(q)}(NH_{\rho})(W,W)(p)
   &-r(q)\frac{C}{K|W|^{2}}\left|NH_{\rho}(W,W)\right|^{2}(p)\\
   &+\mathcal{O}\left(r^{2}(q)\right)|W|^{2}.
 \end{align*}
 Let $c_{1}>0$ be such that $d_{b\Omega}(z)\leq c_{1}|\rho(z)|$ for all $z$ in $V$.
 Then, if we choose
 \begin{align}\label{E:chooseC}
   C\geq\max_{z\in b\Omega, T\in\mathbb{C}^{n}, |T|=1}
   \left\{
     0,\frac{c_{1}\Re\left[(NH_{\rho})(T,T)(z)\right]-\frac{\epsilon}{2}}{\left|(NH_{\rho})(T,T)(z)\right|^{2}}K
   \right\},
 \end{align}
 we obtain, after possibly shrinking $V$,
 \begin{align}\label{E:BasicTaylorrfinal}
   H_{r}(W,W)(q)\geq -\epsilon |r(q)|\cdot|W|^{2}
 \end{align}
 for all $q\in\Omega\cap V$ with $\pi(q)=p$ and $W\in\mathbb{C}^{n}$  in the span of
 $\{W^{\alpha}(p)\}_{\alpha=1}^{n-1-i}$. In fact, after possibly shrinking $V$, \eqref{E:BasicTaylorrfinal} holds with, say, $2\epsilon$ in place of $\epsilon$ for all
 $q\in\Omega\cap V$ satisfying $\pi(q)\in\Sigma_{i}\cap V$ and $W\in\mathbb{C}^{n}$ belonging to
 the span of $\{W^{\alpha}(\pi(q))\}_{\alpha=1}^{n-1-i}$. 
 
 A problem with this construction is that 
 $r$ is not necessarily plurisubharmonic at those weakly pseudoconvex boundary points which are not
 in $\Sigma_{i}$. This possible loss of plurisubharmonicity occurs because the 
 $W^{\alpha}$'s are not necessarily weak directions at those points. This means,
 that we can not simply copy this construction with $r$ in place of $\rho$ to get good estimates near, say, 
 $\Sigma_{i+1}$. Let us be a more explicit. Suppose $\tilde{p}\in\Sigma_{i+1}\cap V$ is such that
 at least one of the $W^{\alpha}$'s is not a weak direction at $\tilde{p}$. That means, if 
 $T$ is a weak complex tangential direction at $\tilde{p}$, then neither does $|\langle\partial\sigma(\tilde{p}),T\rangle|^{2}$ have to be 
 zero nor does $H_{\sigma}(T,T)(\tilde{p})$ have to be non-negative. In view of \eqref{E:explaincutoff}, 
 this says that it might actually happen that
 $(NH_{r})(T,T)(\tilde{p})$ is greater than $(NH_{\rho})(T,T)(\tilde{p})$ for such a vector $T$. That is, by 
 removing the obstruction term at $p$ we might 
 have worsened the situation at $\tilde{p}$. 
 One might think that this does not cause any real problems since we 
 still need to introduce a correcting function $\tilde{\sigma}$ to remove the obstruction to 
 \eqref{E:Main1} on the set $\Sigma_{i+1}\cap V$. However, it might be the case that
 $(NH_{\rho})(T,T)(\tilde{p})=0$. In this case we do not know whether $H_{\tilde{\sigma}}(T,T)$ is 
 strictly positive at $\tilde{p}$, i.e., we do not know whether $H_{\tilde{\sigma}}(T,T)(\tilde{p})$ can make 
 up for any obstructing terms at $\tilde{p}$ introduced by $\sigma$. This says that we need to smoothly 
 cut off $\sigma$ in a manner such that, away from $\Sigma_{i}\cap V$, 
 $|\langle\partial\sigma,T\rangle|^{2}$ stays close to zero and $H_{\sigma}(T,T)$ does not become too
 negative (relative to $\epsilon|T|^{2}$). The construction of such a cut off function will be done in the 
 next section.
 
 \medskip
 
 \section{The cutting off}\label{S:cutoff}

   Let us recall our setting: we are considering a given boundary point $p\in\Sigma_{i}$,  
   $0\leq i \leq n-2$, $V$ a 
   neighborhood of $p$ and smooth, linearly independent $(1,0)$-vector fields 
   $\{W^{\alpha}\}_{\alpha=1}^{n-1-i}$, $\alpha\in\{1,\dots,n-1-i\}$ on $V$, which are complex tangential
   to $b\Omega$ on $b\Omega\cap V$ and satisfy $H_{\rho}(W^{\alpha},W^{\alpha})=0$ on 
   $\Sigma_{i}\cap V$. From now on, we also suppose that $V$ and the $W^{\alpha}$'s are 
   chosen such that the 
   span of $\{W^{\alpha}(z)\}_{\alpha=1}^{n-1-i}$ contains the null space of the Levi form of $\rho$ at $z$ 
   for all $z\in\Sigma_{j}\cap V$ for $j\in\{i+1,\dots,n-2\}$. This can be done by first selecting
   smooth $(1,0)$-vector fields $\{S^{\beta}(z)\}_{\beta=1}^{i}$ which are complex tangential to
   $b\Omega$ on $b\Omega\cap V$ for some neighborhood $V$ of $p$ and orthogonal to each other
    with respect to the Levi form of $\rho$ such that $H_{\rho}(S^{\beta},S^{\beta})>0$ holds on
    $b\Omega\cap V$ after possibly shrinking $V$. Then one completes the basis of the complex
    tangent space with smooth $(1,0)$-vector fields $\{W^{\alpha}(z)\}_{\alpha=1}^{n-1-i}$ such that
    the $W^{\alpha}$'s are orthogonal to the $S^{\beta}$'s with respect to the Levi form of $\rho$.
   
   Let $V'\ssubset V$ be another neighborhood of $p$. Let $\zeta\in C^{\infty}_{c}(V,[0,1])$ be a function
    which equals $1$ on $V'$. For given $m>2$, let $\chi_{m}\in C^{\infty}(\mathbb{R})$ be
    an increasing function with $\chi_{m}(x)=1$ for all $x\leq 1$ and $\chi_{m}(x)=e^{m}$ for all
    $x\geq e^{m}$ such that
     \begin{align*}
      \frac{x}{\chi_{m}(x)}\leq 2,\;\;
      \chi_{m}'(x)\leq 2,\;\text{and}\;\;
      x\cdot\chi_{m}''(x)\leq 4\;\;\text{for all}\;\; x\in[1,e^{m}].
    \end{align*}
    Set $\chi_{m,\tau}(x)=\chi_{m}\left(\frac{x}{\tau}\right)$ for given $\tau>0$. The above properties then 
    become
     \begin{align*}
      \frac{x}{\chi_{m,\tau}(x)}\leq 2\tau,\;\;
      \chi_{m,\tau}'(x)\leq \frac{2}{\tau},\;\text{and}\;\;
      x\cdot\chi_{m,\tau}''(x)\leq \frac{4}{\tau}\;\;\text{for all}\;\; x\in[\tau,\tau e^{m}].
    \end{align*}
   Set $g_{m,\tau}(x)=1-\frac{\ln(\chi_{m,\tau}(x))}{m}$. It follows by a straightforward computation that
   \begin{align}\label{E:propsg}
     &g_{m,\tau}(x)=1\;\text{for}\;x\leq\tau,\;\;0\leq g_{m,\tau}(x)\leq 1\;\text{for all}\;x\in\mathbb{R},
     \;\text{and}\notag\\
     &|g_{m,\tau}'(x)|\leq\frac{4}{m}\cdot\frac{1}{x},\;\;g_{m,\tau}''(x)\geq-\frac{8}{m}\cdot\frac{1}{x^{2}}\;\;
     \text{for}\;x\in(\tau, \tau e^{m}).
   \end{align}
   For given $m,\;\tau> 0$ we define
   \begin{align*}
     s_{m,\tau}(z)=\zeta(z)\cdot\sigma(z)\cdot g_{m,\tau}(\sigma(z))\;\;\text{for}\;z\in V
   \end{align*}
   and $s_{m,\tau}=0$ outside of $V$. This function has the properties described at the
   end of Section \ref{S:modification} if $m,\tau$ are chosen appropriately:
   \begin{lemma}\label{L:cutoff}
     For all $\delta>0$, there exist $m,\;\tau>0$ such that 
     $s_{m,\tau}$ satisfies:
     \begin{enumerate}
       \item[(i)] $s_{m,\tau}=\zeta\sigma$ for $\sigma\in[0,\tau]$,
       \item[(ii)] $0\leq s_{m,\tau}\leq\delta$ on $b\Omega$.
     \end{enumerate}
     Moreover, if $z\in \left(b\Omega\cap V\right)\setminus\Sigma_{i}$ and 
     $T\in\mathbb{C}T_{z}b\Omega$, then
     \begin{enumerate}
       \item[(iii)] $|\langle\partial s_{m,\tau}(z),T\rangle|\leq\delta|T|$,
       \item[(iv)] $H_{s_{m,\tau}}(T,T)(z)\geq-\delta|T|^{2}$ if $T\in\Span\{W^{\alpha}(z)\}$.
     \end{enumerate}
   \end{lemma}
   
   Note that part (iv) of Lemma \ref{L:cutoff} in particular says that if 
   $z\in\cup_{j=i+1}^{n-2}\Sigma_{j}\cap V$, then $H_{s_{m,\tau}}(T,T)(z)\geq-\delta|T|^{2}$ for all $T$ 
   which are weak complex tangential vectors at $z$.
   
   To prove Lemma \ref{L:cutoff}, we will need $|\langle\partial\sigma,T\rangle|^{2}
   \lesssim\sigma|T|^{2}$ on
   $\overline{\supp(\zeta)}\cap b\Omega$. That this is in fact true we learned from J. D. McNeal.
   
  \begin{lemma}[\cite{McN}]\label{L:McNeal}
     Let $U\ssubset\mathbb{R}^{n}$ be open. Let $f\in C^{2}(U)$ be a non-negative  
     function on $U$. Then for any compact set $K\ssubset U$, there exists a constant $c>0$ such that
     \begin{align}\label{E:McNeal}
       |\nabla f(x)|^{2}\leq c f(x)\;\;\text{for all}\;\;x\in K.
     \end{align}
   \end{lemma}
   Since the proof by McNeal is rather clever, and since we are not aware of it being published, 
   we shall give it here.
   \begin{proof}
     Let $F$ be a smooth, non-negative function such that $F=f$ on $K$ and $F=0$ on 
     $\mathbb{R}^{n}\setminus U$.
     For a given $x\in K$, we have for all $h\in\mathbb{R}^{n}$ that
     \begin{align}\label{E:McNealTaylor}
       0\leq F(x+h)&=F(x)+\sum_{k=1}^{n}\frac{\partial F}{\partial x_{k}}(x)h_{k}
       +\frac{1}{2}\sum_{k,l=1}^{n}\frac{\partial^{2} F}{\partial x_{k}\partial x_{l}}(\xi)h_{k}h_{l}\notag\\
       &=f(x)+\sum_{k=1}^{n}\frac{\partial f}{\partial x_{k}}(x)h_{k}+
       \frac{1}{2}\sum_{k,l=1}^{n}\frac{\partial^{2} F}{\partial x_{k}\partial x_{l}}(\xi)h_{k}h_{l}
     \end{align}
     holds for some $\xi\in U$.
     Note that \eqref{E:McNeal} is true if $\left(\nabla f\right)(x)=0$. So assume now that
     $\left(\nabla f\right)(x)\neq 0$ and  choose
     $h_{k}=\frac{\frac{\partial f}{\partial x_{k}}(x)}{\left|\left(\nabla f\right)(x)\right|}\cdot t$ for
     $t\in\mathbb{R}$. Then
     \eqref{E:McNealTaylor} becomes
     \begin{align*}
       0\leq f(x)+\left|\left(\nabla f\right)(x)\right|t
       +nL\cdot\frac{\sum_{k=1}^{n}\left|\frac{\partial f}{\partial x_{k}}(x)\right|^{2}}{\left|\left(\nabla f\right)
       (x) \right|^{2}}
       \cdot t^{2}\;\;\text{for all}\;\;t\in\mathbb{R},
     \end{align*}
     where $L=\frac{1}{2}\sup
     \left\{\left|\frac{\partial^{2}F}{\partial x_{k}\partial x_{l}}(\xi)
     \right|\;|\;\xi\in U,\;1\leq k,l\leq n\right\}$.  
   Therefore, \eqref{E:McNealTaylor} becomes
   \begin{align*}
     0\leq f(x)+\left|\left(\nabla f\right)(x)\right|\cdot t+nL\cdot t^{2}\;\;\text{for all}\;\;t\in\mathbb{R}.
   \end{align*}
   In particular, the following must hold for all 
   $t\in\mathbb{R}$:
   \begin{align*}
     -\frac{f(x)}{nL}+\frac{\left|\left(\nabla f\right)(x)\right|^{2}}{(2nL)^{2}}
     \leq
     \left(
     t+\frac{\left|\left(\nabla f\right)(x)\right|}{2nL}
     \right)^{2},
   \end{align*}
   which implies that
   \begin{align*}
     \left|\left(\nabla f\right)(x)\right|^{2}\leq 4nL\cdot f(x).
   \end{align*}
 \end{proof}
   We can assume that $V$ is such that there exists a diffeomorphism $\phi:V\cap b\Omega
   \longrightarrow U$ for some open set $U\ssubset\mathbb{R}^{2n-1}$. 
   Set $f:=\sigma\circ\phi^{-1}$. Then $f$ satisfies the
   hypotheses of  Lemma \ref{L:McNeal}. Hence we get that there exists a constant $c>0$ such that
   $\left|\left(\nabla f\right)(x)\right|^{2}\leq cf(x)$ for all $x\in K=\phi(\overline{\supp(\zeta)})$. This implies 
   that there exists a constant $c_{1}>0$, depending on $\phi$, such that
   \begin{align}\label{E:McNealapplied}
     \left|\langle\partial\sigma(z),T\rangle\right|^{2}\leq c_{1}\sigma(z)|T|^{2}\;\;
     \text{for all}\;\; z\in\overline{\supp(\zeta)}\;\text{and}\; T\in\mathbb{C}T_{z}b\Omega.
   \end{align}
   Now we can prove Lemma \ref{L:cutoff}.
   \begin{proof}[Proof of Lemma \ref{L:cutoff}]
     Note first that $s_{m,\tau}$ is identically zero 
     on $b\Omega\setminus V$ for any $m>2$ and $\tau>0$.

     Now let $\delta>0$ be given, let $m$ be a large, positive number, fixed and to be chosen later 
     (that is, in the proof of (iv)). 
     Below we will show how to 
     choose $\tau>0$ once $m$ has been chosen.
     
     Part (i) follows directly from the definition of $s_{m,\tau}$ for any choice of $m>2$ and $\tau>0$. 
     Part (ii) also follows 
     straightforwardly, if $\tau>0$ is such that $\tau e^{m}\leq\delta$.
     Notice that for all $z\in b\Omega\cap V$ with $\sigma(z)>\tau e^{m}$, $s_{m,\tau}(z)=0$, and hence  
     (iii), (iv) hold trivially there. Thus, to
     prove (iii) and (iv) we only need to consider the two sets
     \begin{align*}
       S_{1}=\{z\in b\Omega\cap V\;|\;\sigma(z)\in(0,\tau)\}\;\;\text{and}\;\;
       S_{2}=\{z\in b\Omega\cap V\;|\;\sigma(z)\in[\tau,\tau e^{m}]\}.
     \end{align*}
     
     \medskip
     
     \emph{Proof of (iii):} If $z\in S_{1}$, then $s_{m,\tau}(z)=\zeta(z)\cdot\sigma(z)$ and
     if $T\in\mathbb{C}T_{z}b\Omega$, we get
     \begin{align*}
       \left|\langle\partial s_{m,\tau}(z),T\rangle
       \right|&=
       \left|
       \sigma\cdot\langle\partial\zeta,T\rangle+\zeta\cdot\langle\partial\sigma, T\rangle
       \right|(z)\\
       &\overset{\eqref{E:McNealapplied}}{\leq} \left(c_{2}\sigma(z)
       +\left(c_{1}\sigma(z)\right)^{\frac{1}{2}}\right)|T|, \;
       \text{where}\;\; c_{2}=\max_{z\in b\Omega\cap V}|\partial\zeta(z)|.
     \end{align*}
     Thus, if we choose $\tau>0$ such that, $c_{2}\tau+(c_{1}\tau)^{\frac{1}{2}}\leq\delta$, then
     (iii) holds on the set $S_{1}$.\\
     Now suppose that $z\in S_{2}, T\in\mathbb{C}T_{z}b\Omega$ and compute:
     \begin{align*}
       \left|
       \langle\partial s_{m,\tau}(z),T\rangle
       \right|
       &=\left|
        \sigma\cdot g_{m,\tau}(\sigma)\cdot\langle\partial\zeta,T\rangle
        +\zeta\cdot\langle\partial\sigma,T\rangle
        \left(
        g_{m,\tau}(\sigma)+\sigma\cdot g_{m,\tau}'(\sigma)
        \right)
       \right|(z)\\
       &\overset{\eqref{E:McNealapplied},\eqref{E:propsg}}{\leq}
       \left(
       c_{2}\sigma(z)+\left(c_{1}\sigma(z)\right)^{\frac{1}{2}}\cdot\left(1+\frac{4}{m}\right)
       \right)|T|.
     \end{align*}
     Thus, if we choose $\tau>0$ such that $c_{2}\tau e^{m}+2(c_{1}\tau e^{m})^{\frac{1}{2}}
     \leq\delta$, then (iii) also holds on the set $S_{2}$.
     
     \medskip
     
     \emph{Proof of (iv):} Let us first consider the case when $z\in S_{1}$. Then, again,
     $s_{m,\tau}(z)=\zeta(z)\cdot\sigma(z)$ and if $T$ is in the span of 
     $\{W^{\alpha}(z)\}_{\alpha=1}^{n-1-i}$, we obtain
     \begin{align*}
       H_{s_{m,\tau}}(T,T)(z)=\left[\sigma\cdot H_{\zeta}(T,T)
       +2\Re\left(\langle\partial\zeta,T\rangle\cdot
       \overline{\langle\partial\sigma,T\rangle}\right)
       +\zeta\cdot H_{\sigma}(T,T)
       \right](z).
     \end{align*}
     Let $c_{3}>0$ be a constant such that $H_{\zeta}(\xi,\xi)(z)\geq 
     -c_{3}|\xi|^{2}$ for all  $z\in b\Omega\cap V$, $\xi\in\mathbb{C}^{n}$. Then it follows, using 
     \eqref{E:McNealapplied} again, that
     \begin{align*}
       H_{s_{m,\tau}}(T,T)(z)\geq
       -\left(c_{3}\sigma(z)+2c_{2}\left(c_{1}\sigma(z) \right)^{\frac{1}{2}}\right) |T|^{2}
       +\zeta(z)\cdot H_{\sigma}(T,T)(z).
     \end{align*}
     Note that  for $z$ and $T$ as above, $H_{\sigma}(T,T)(z)\geq 0$ when $\sigma(z)=0$. 
     Furthermore, the set 
     \begin{align*}
       \left\{(z,T)\;|\;z\in b\Omega\cap\overline{\supp\zeta},\;\sigma(z)=0,\;
       T\in\Span\{W^{1}(z),\dots,W^{n-1-i}(z)\}\right\}
     \end{align*}
      is a closed subset of the complex tangent bundle of $b\Omega$. 
      Thus there exists a neighborhood $U\subset V$
     of $\{z\in b\Omega\cap\overline{\supp\zeta}\;|\;\sigma(z)=0\}$ such that
     \begin{align*}
       H_{\sigma}(T,T)(z)\geq-\frac{\delta}{2}|T|^{2}
     \end{align*}
     holds for all $z\in b\Omega\cap U$, $T$ in the span of $\{W^{\alpha}(z)\}_{\alpha=1}^{n-1-i}$.
     Let $\nu_{1}$ be the maximum of $\sigma$ on the closure of $b\Omega\cap U$, and let
     $\nu_{2}$ be the minimum of $\sigma$ on 
     $\left(b\Omega\cap\overline{\supp\zeta}\right)\setminus U$.
     Now choose $\tau>0$ such that
     $\tau \leq\min\{\nu_{1},\frac{\nu_{2}}{2}\}$. Then $z\in S_{1}$ implies that $z\in b\Omega\cap U$ and 
     therefore  
     $\zeta(z)\cdot H_{\sigma}(T,T)(z)\geq-\frac{\delta}{2}|T|^{2}$ for all $T$ in the span of
     the $W^{\alpha}(z)$'s.
     If we also make sure that $\tau>0$ is such that $c_{3}\tau+2c_{2}\left(c_{1}\tau\right)^{\frac{1}{2}}
     \leq\frac{\delta}{2}$, then (iv) is true on $S_{1}$.
     
     Now suppose that $z\in S_{2}$ and $T$ in the span of 
     $\{W^{\alpha}(z)\}_{\alpha=1}^{n-1-i}$. We compute
     \begin{align*}
       H_{s_{m,\tau}}&(T,T)(z)=
       \Bigl[\sigma g_{m,\tau}(\sigma) H_{\zeta}(T,T)
       +2\Re\left(\langle\partial\zeta,T\rangle\cdot
       \overline{\langle\partial\sigma,T\rangle}
       \right)
       \left(
       g_{m,\tau}(\sigma)+\sigma g_{m,\tau}'(\sigma)
       \right)\Bigr.\\
       &\Bigl.+
       \zeta\left|\langle\partial\sigma,T\rangle\right|^{2}
       \left(2g_{m,\tau}'(\sigma)+\sigma g_{m,\tau}''(\sigma)\right)
       +\zeta H_{\sigma}(T,T)
       (g_{m,\tau}(\sigma)+\sigma g_{m,\tau}'(\sigma))
       \Bigr](z)\\
       &=\I+\II+\III+\IV.
     \end{align*}
     If we choose $\tau>0$ such that $\tau e^{m}c_{3}\leq\frac{\delta}{4}$, then it follows that
     $\I\geq-\frac{\delta}{4}|T|^{2}$. Estimating the term $\II$ we get
     \begin{align*}
       \II\overset{\eqref{E:propsg}}{\geq}-2c_{2}\left|\langle\partial\sigma(z),T\rangle
       \right||T|\left(1+\frac{4}{m}\right)
       \overset{\eqref{E:McNealapplied}}{\geq} -4c_{2}\left(c_{1}\sigma(z)\right)^{\frac{1}{2}}|T|^{2}
       \geq-\frac{\delta}{4}|T|^{2},
     \end{align*}
     if we choose $\tau>0$ such that $4c_{2}\left(c_{1}\tau e^{m}\right)^{\frac{1}{2}}\leq\frac{\delta}{4}$.
     To estimate term
     $\IV$, we only need to make sure that $\tau>0$ is so small that $z\in S_{2}$ implies that 
     $2\zeta(z)\cdot H_{\sigma}(T,T)(z)\geq
     -\frac{\delta}{4}|T^{2}$. This can be done similarly to the case when $z\in S_{1}$.
     
     Note that up to this point the size of the parameter $m$ played no role. That is, we obtain above 
     results for any choice of $m$ as long as $\tau>0$ is sufficiently small. The size of $m$ only matters
     for the estimates on term $\III$:  \eqref{E:propsg} and \eqref{E:McNealapplied} yield
     \begin{align*}
       \III&\geq -\left|\langle\partial\sigma(z),T\rangle\right|^{2}
       \left[2|g_{m,\tau}'(\sigma(z))|+\sigma(z)g_{m,\tau}''(\sigma(z))\right]\\
       &\geq -\left|\langle\partial\sigma(z),T\rangle\right|^{2}\frac{16}{m\sigma(z)}
       \geq-\frac{16 c_{1}}{m}|T|^{2}.
     \end{align*}
     We now choose $m>0$ such that  $\frac{16 c_{1}}{m}\leq\frac{\delta}{4}$, and then we choose 
     $\tau>0$ according to our previous computations .   
   \end{proof}
  
  \medskip
  
  \section{Proof of  \eqref{E:Main1}}\label{S:proof}
    We shall prove \eqref{E:Main1} by induction over the rank of the Levi form of $\rho$. To start the 
    induction we construct a smooth defining function $r_{0}$  of $\Omega$ which satisfies 
    \eqref{E:Main1} on $U_{0}\cap\Omega$ for some neighborhood $U_{0}$ of $\Sigma_{0}$. 
    
    Let $\{V_{j,0}\}_{j\in J_{0}}$, $\{V_{j,0}'\}_{j\in J_{0}}\ssubset\mathbb{C}^{n}$ be 
    finite, open covers of $\Sigma_{0}$ with $V_{j,0}'\ssubset V_{j,0}$ such that there exist smooth, 
    linearly independent $(1,0)$-vector fields $W_{j,0}^{\alpha}$, $\alpha\in\{1,\dots,n-1\}$, 
    defined on $V_{j,0}$, which
    are complex tangential to $b\Omega$ on $b\Omega\cap V_{j,0}$ and satisfy:
    \begin{enumerate}
      \item $H_{\rho}(W_{j,0}^{\alpha},W_{j,0}^{\alpha})=0$ on $\Sigma_{0}\cap V_{j,0}$ for all
         $j\in J_{0}$,
      \item  the span of the $W_{j,0}^{\alpha}(z)$'s contains the null space of the Levi form of
          $\rho$ at $z$ for all boundary points $z$ in $\cup_{j=i+1}^{n-2}\Sigma_{j}\cap V_{j,0}$.
    \end{enumerate}
     We shall write $V_{0}=\cup_{j\in J_{0}}V_{j,0}$ 
    and $V_{0}'=\cup_{j\in J_{0}}V_{j,0}'$. 
    
    Choose smooth, non-negative functions
    $\zeta_{j,0}$, $j\in J_{0}$, such that
    \begin{align*}
      \sum_{j\in J_{0}}\zeta_{j,0}=1\;\text{on}\;V_{0}',\;\sum_{j\in J_{0}}\zeta_{j,0}\leq 1\;\text{on}\;V_{0},
      \;\text{and}\;\overline{\supp\zeta_{j,0}}\subset V_{j,0}\;\text{for all}\;j\in J_{0}.
    \end{align*}
   Set $\sigma_{j,0}=\sum_{\alpha=1}^{n-1}H_{\rho}(W_{j,0}^{\alpha},W_{j,0}^{\alpha})$. For given 
   $\epsilon>0$, choose $C_{0}$ according to $\eqref{E:chooseC}$. Then choose $m_{j,0}$
   and $\tau_{j,0}$ as in Lemma \ref{L:cutoff} such that
   \begin{align*}
    s_{m_{j,0},\tau_{j,0}}(z)=
    \begin{cases}
         \zeta_{j,0}(z)\cdot\sigma_{j,0}(z)\cdot g_{m_{j,0},\tau_{j,0}}(\sigma_{j,0}(z)) &\text{if}\;\;z\in V_{j,0}\\
         0 &\text{if}\;\;z\in (V_{j,0})^{c}
    \end{cases}
   \end{align*}
  satisfies (i)-(iv) of Lemma \ref{L:cutoff} for $\delta_{0}=\frac{\epsilon}{C_{0}|J_{0}|}$.  
  Finally, set $s_{0}=\sum_{j\in J_{0}}s_{m_{j,0},\tau_{j,0}}$ and
  define the smooth defining function $r_{0}=\rho e^{-C_{0}s_{0}}$.
  
  \medskip
  
  By our choice of $r_{0}$ we have for all 
  $q\in\Omega\cap V_{0}'$ with $\pi(q)\in\Sigma_{0}\cap V_{0}'$  that
  \begin{align*}
    H_{r_{0}}(W,W)(q)\geq H_{r_{0}}(W,W)(\pi(q))-\epsilon \left|r_{0}(q)\right|\cdot |W|^{2}
    =-\epsilon \left|r_{0}(q)\right|\cdot |W|^{2}
   \end{align*}  
  for all $W\in\mathbb{C}T_{\pi(q)}b\Omega$.
  In fact, by continuity there exists a neighborhood $U_{0}\subset V_{0}'$ of $\Sigma_{0}$ such that
  \begin{align*}
    H_{r_{0}}(W,W)(q)\geq H_{r_{0}}(W,W)(\pi(q))-2\epsilon |r_{0}(q)||W|^{2}
  \end{align*}
  holds for all $q\in\Omega\cap U_{0}$ with $\pi(q)\in b\Omega\cap U_{0}$ and
  $W\in\mathbb{C}T_{\pi(q)}b\Omega$.
  
  Now let $\xi\in\mathbb{C}^{n}$. For each $q\in\Omega\cap U_{0}$ 
  with $\pi(q)\in b\Omega\cap U_{0}$ we shall
  write $\xi=W+M$, where $W\in\mathbb{C}T_{\pi(q)}b\Omega$ and $M$ in the span of $N(\pi(q))$.
  Note that then $|\xi|^{2}=|W|^{2}+|M|^{2}$.
  We get for the complex Hessian of $r_{0}$ at $q$:
  \begin{align*}
    H_{r_{0}}(\xi,\xi)(q)&=H_{r_{0}}(W,W)(q)+2\Re\left(
    H_{r_{0}}(W,M)(q)
    \right)
    +H_{r_{0}}(M,M)(q)\\
    &\geq H_{r_{0}}(W,W)(\pi(q))-2\epsilon\left|r_{0}(q)\right|\cdot |W|^{2}+
    2\Re\left(
    H_{r_{0}}(W,M)(q)
    \right)
    +H_{r_{0}}(M,M)(q).
  \end{align*}
  Note that Taylor's theorem yields
  \begin{align*}
    H_{r_{0}}(W,M)(q)&=H_{r_{0}}(W,M)(\pi(q))+\mathcal{O}(d_{b\Omega}(q))|W||M|\\
    &=
    e^{-C_{0}s_{0}(\pi(q))}\left(H_{\rho}(W,M)-C\overline{\langle\partial\rho,M\rangle}
    \langle\partial s_{0},W\rangle\right)(\pi(q))+\mathcal{O}(d_{b\Omega}(q))|W||M|.
  \end{align*}
   It follows by property (iii) of Lemma \ref{L:cutoff} that
  $|\langle\partial s_{0},W\rangle|\leq\frac{\epsilon}{C_{0}}|W|$ on $b\Omega$.
  After possibly shrinking $U_{0}$ we get
  \begin{align*}
    2\Re\left(H_{r_{0}}(W,M)(\pi(q))\right)\geq-4\epsilon|\partial\rho| |W| |M|
    +e^{-C_{0}s_{0}(\pi(q))}2\Re\left(H_{\rho}(W,M)\right)(\pi(q)).
  \end{align*}
  Putting the above estimates together, we have 
  \begin{align*}
    H_{r_{0}}(\xi,\xi)(q)\geq
    &-2\epsilon|r_{0}(q)||W|^{2}-4\epsilon|\partial\rho||W||M|+H_{r_{0}}(M,M)(q)\\
    &+e^{-C_{0}s_{0}(\pi(q))}
    \left[
      H_{\rho}(W,W)(\pi(q))+2\Re\left(H_{\rho}(W,M)(\pi(q))\right)
    \right].
  \end{align*}
 An application of the (sc)-(lc) inequality yields
   \begin{align*} 
   -\epsilon |W||M|\geq-\epsilon\left(|r_{0}(q)||\xi|^{2}
   +\frac{1}{|r_{0}(q)|}\left|\langle\partial r_{0}(\pi(q)),\xi\rangle\right|^{2}
   \right),
  \end{align*}
  where we used that $|\xi|^{2}=|W|^{2}+|M|^{2}$. Taylor's
  theorem also gives us that 
  \begin{align*}
    \left|\langle\partial r_{0}(\pi(q)),\xi\rangle\right|
    &=e^{-C_{0}s_{0}(\pi(q))}\cdot\left|
    \langle\partial\rho(\pi(q)),\xi\rangle
    \right|\\
    &\leq e^{-C_{0}s_{0}(\pi(q))}\cdot\left(
    \left|\langle\partial\rho(q),\xi\rangle\right|+
    \mathcal{O}\left(\rho(q)\right)|\xi|
    \right),
  \end{align*}
  where the constant in the last term is independent of $\epsilon$. After possibly shrinking $U_{0}$, we 
  obtain for all $q\in\Omega\cap U_{0}$
  \begin{align*}
     \left|\langle\partial r_{0}(\pi(q)),\xi\rangle\right|
     &\leq 2e^{-C_{0}s_{0}(q)}\cdot\left|\langle\partial\rho(q),\xi\rangle\right|
     +\mathcal{O}\left(r_{0}(q)\right)|\xi|\\
     &\leq 2\left|\langle\partial r_{0}(q),\xi\rangle\right|+2\left|\rho(q)\right|
     \cdot\left|\langle\partial e^{-C_{0}s_{0}(q)},\xi\rangle\right|
     +\mathcal{O}\left(r_{0}(q)\right)|\xi|,
  \end{align*}
  where, again, the constant in the last term is independent of $\epsilon$. Using that 
  $\left|\langle\partial s_{0},W\rangle\right|\leq\frac{\epsilon}{C_{0}}|W|$ on $b\Omega$, we get
  \begin{align*}
    \left|\langle\partial e^{-C_{0}s_{0}(q)},\xi\rangle\right|
    \leq 2\epsilon |W|+\mathcal{O}\left( \left|\langle\partial\rho(\pi(q)),\xi\rangle  \right| \right),
  \end{align*}
  which implies that
  $|\langle\partial r_{0}(\pi(q)),\xi\rangle|\lesssim|\langle\partial r_{0}(q),\xi\rangle|+
  |r_{0}(q)||\xi|$. Thus we have
  \begin{align*}
    -\epsilon|W||M|\gtrsim-\epsilon\left(|r_{0}(q)||\xi|^{2}
   +\frac{1}{|r_{0}(q)|}\left|\langle\partial r_{0}(q),\xi\rangle\right|^{2}
   \right).
  \end{align*}
  Since 
  \begin{align*}
    H_{\rho}(W,W)(\pi(q))+2\Re(H_{\rho}(W,M)(\pi(q))=H_{\rho}(\xi,\xi)(\pi(q))-H_{\rho}(M,M)(\pi(q)),
  \end{align*}
  and $H_{\rho}(\xi,\xi)(\pi(q))$ is non-negative, it follows that   
  \begin{align*}
    H_{r_{0}}(\xi,\xi)(q)\gtrsim -\epsilon\left(|r_{0}(q)||\xi|^{2}+\frac{1}{|r_{0}(q)|}
    \left|\langle\partial r_{0}(q),\xi\rangle\right|^{2}\right)
    +\mu_{0}H_{\rho}(\xi,\xi)(\pi(q))
  \end{align*}
  for some positive constant $\mu_{0}$.
  Since the constants in $\gtrsim$ do not depend on the choice of $\epsilon$, this
  proves \eqref{E:Main1} in an open neighborhood $U_{0}$ of $\Sigma_{0}$.
  
  \medskip
  
  Let $l\in\{0,\dots,n-3\}$ be fixed and suppose that there exist a smooth defining function
  $r_{l}$ of $b\Omega$ and an open neighborhood $U_{l}\subset\mathbb{C}^{n}$ of 
  $\cup_{i=0}^{l}\Sigma_{i}$ such that
  \begin{align}\label{E:ihypotheses1}
    H_{r_{l}}(\xi,\xi)(q)\geq -\epsilon
    \left(
    |r_{l}(q)||\xi|^{2}+\frac{1}{|r_{l}(q)|}\left|\langle\partial r_{l}(q),\xi\rangle\right|^{2}
    \right)+\mu_{l}H_{\rho}(\xi,\xi)(\pi(q))    
  \end{align}
  for all $q\in\Omega\cap U_{l}$ with $\pi(q)\in b\Omega\cap U_{l}$ and $\xi\in\mathbb{C}^{n}$. Here, 
  $\mu_{l}$ is some positive constant.
  Furthermore, we suppose that the function $\vartheta_{l}$ defined by $r_{l}=\rho e^{-\vartheta_{l}}$  
  satisfies 
  the following
  \begin{align}
    \left|
    \langle\partial\vartheta_{l}(z),T\rangle\right|&\leq\epsilon|T|\;\;\text{for all}\;\;z\in b\Omega, T\in
    \mathbb{C}T_{z}b\Omega\label{E:ihypothesis2}\\
    H_{\vartheta_{l}}(T,T)(z)&\geq-\epsilon|T|^{2}\;\;\text{for all}\;\; z\in\cup_{j=l+1}^{n-2}\Sigma_{j},
    T\in\mathbb{C}T_{z}b\Omega\;\;\text{with}\;\;H_{\rho}(T,T)(z)=0\label{E:ihypothesis3}.
  \end{align} 
  
  Let $k=l+1$. We shall now show that there exist a smooth defining function $r_{k}$ and
  a neighborhood $U_{k}$ of $\cup_{i=0}^{k}\Sigma_{i}$ such that for some positive constant
  $\mu_{k}$
  \begin{align}\label{E:claimistep}
    H_{r_{k}}(\xi,\xi)(q)\geq -\epsilon\left(
    |r_{k}(q)||\xi|^{2}+\frac{1}{|r_{k}(q)|}\left|
    \langle\partial r_{k}(q),\xi\rangle
    \right|^{2}
    \right)+\mu_{k}H_{\rho}(\xi,\xi)(\pi(q))
  \end{align}
  holds for all $q\in \Omega\cap U_{k}$ with $\pi(q)\in b\Omega\cap U_{k}$ and $\xi\in\mathbb{C}^{n}$.
  
  \medskip
  
  Let $\{V_{j,k}\}_{j\in J_{k}}$, $\{V_{j,k}'\}_{j\in J_{k}}$ be finite, open covers of $\Sigma_{k}\setminus
  U_{k-1}$ such that 
  \begin{enumerate}
  \item[(1)]   $V_{j,k}'\ssubset V_{j,k}$ and $\overline{V}_{j,k}
    \cap\left(\cup_{i=0}^{k-1}\Sigma_{i}\right)=\emptyset$ 
  for all $j\in J_{k}$, and 
  \item[(2)] there exist smooth, linearly 
  independent $(1,0)$-vector fields $W_{j,k}^{\alpha}$, $\alpha\in\{1,\dots,n-1-k\}$, 
  and $S_{j,k}^{\beta}$, $\beta\in\{1,\dots,k\}$, defined on $V_{j,k}$, which are 
  complex tangential to $b\Omega$ on $b\Omega\cap V_{j,k}$ and satisfy the following:
  \begin{enumerate}
    \item[(a)]  $H_{\rho}(W_{j,k}^{\alpha},W_{j,k}^{\alpha})=0$ on $\Sigma_{k}\cap V_{j,k},\;
      \alpha\in\{1,\dots,n-1-k\},\;j\in J_{k}$,
    \item[(b)] the span of $\{W_{j,k}^{1}(z),\dots,W_{j,k}^{k}(z)\}$ contains the null space of the Levi form
      of $\rho$ at all boundary points $z$ belonging to $\cup_{i=k+1}^{n-2}\Sigma_{i}\cap V_{j,k}$,  
    \item[(c)] $H_{\rho}(S_{j,k}^{\beta}, S_{j,k}^{\beta})> 0$ on $b\Omega\cap \overline{V_{j,k}},\;
     \beta\in\{1,\dots,k\},\;j\in J_{k}$,
   \item[(d)]  $H_{\rho}(S_{j,k}^{\beta},S_{j,k}^{\tilde{\beta}})=0$ for $\beta\neq\tilde{\beta}$ on
   $b\Omega\cap V_{j,k}$, $\beta,\tilde{\beta}\in\{1,\dots,k\},\;j\in J_{k}$.  
  \end{enumerate}
  \end{enumerate}
  Note that above vector fields $\{W_{j,k}^{\alpha}\}$ 
  always exist in some neighborhood of a given point in
  $\Sigma_{k}$. However, we might not be able to cover $\Sigma_{k}$ with finitely many such 
  neighborhoods, when the closure of $\Sigma_{k}$ contains boundary points at which the Levi 
  form is of lower rank. Moreover, if the latter is the case, then (c) above is also impossible. These are the
   reasons for proving \eqref{E:Main1} via induction over the rank of the Levi form of 
  $\rho$. 
  
   Suppose $S$ is in the span of $\{S_{j,k}^{\beta}(z)\}_{\beta=1}^{k}$ and
   $W$ is in the span of $\{W_{j,k}^{\alpha}(z)\}_{\alpha=1}^{n-1-k}$
   for some $z\in V_{j,k}$,    then there is some constant
  $\kappa_{k}>0$ such that $|S|^{2}+|W|^{2}\leq\kappa_{k}|S+W|^{2}$ for all $j\in J_{k}$. 
  We shall write $V_{k}=\cup_{j\in J_{k}}V_{j,k}$ and $V_{k}'=\cup_{j\in J_{k}}V_{j,k}'$.
  
  \medskip
  
  Let $\zeta_{j,k}$ be non-negative, smooth functions such that
  \begin{align*}
    \sum_{j\in J_{k}}\zeta_{j,k}=1\;\text{on}\;V_{k}',\;\sum_{j\in J_{k}}\zeta_{j,k}\leq 1\;
     \text{on}\;V_{k}\;\text{and}\;\overline{\supp\zeta_{j,k}}\subset V_{j,k}. 
  \end{align*}
  Set $\sigma_{j,k}=\sum_{\alpha=1}^{n-k-1}H_{\rho}(W_{j,k}^{\alpha},W_{j,k}^{\alpha})$. Recall that
  $\epsilon>0$ is given. Choose $C_{k}$  according to \eqref{E:chooseC} with 
  $\frac{\epsilon}{\kappa_{k}}$ in place of $\epsilon$ there.
  We now choose $m_{j,k},\;\tau_{j,k}>0$ such that
  \begin{align*}
    s_{m_{j,k},\tau_{j,k}}(z)=
    \begin{cases}
      \zeta_{j,k}(z)\cdot\sigma_{j,k}(z)\cdot g_{m_{j,k}\tau_{j,k}}(\sigma_{j,k}(z))&\text{if}\;
      z\in V_{j,k}\\
      0&\text{if}\;z\in (V_{j,k})^{c}
    \end{cases}
  \end{align*}
  satisfies (i)-(iv) of Lemma \ref{L:cutoff}  with
  $\delta_{k}=\frac{\epsilon}{C_{k}|J_{k}|\kappa_{k}}$. Set $s_{k}=\sum_{j\in J_{k}}s_{m_{j,k},\tau_{j,k}}$ 
  and define the smooth defining function
  $r_{k}=r_{k-1}e^{-Cs_{k}}$. We claim that this choice of $r_{k}$ satisfies
  \eqref{E:claimistep}.
  We shall first see that \eqref{E:claimistep} is true for all $q\in\Omega\cap U_{k-1}\cap V_{k}$ 
  with $\pi(q)\in b\Omega\cap U_{k-1}\cap V_{k}$.
  A straightforward computation yields
  \begin{align}\label{E:istepint1}
    H_{r_{k}}(\xi,\xi)(q)=
    e^{-C_{k}s_{k}(q)}
    \biggl[
      H_{r_{k-1}}(\xi,\xi)\biggr.&+r_{k-1}\left(
      C_{k}^{2}\left|\langle\partial s_{k},\xi\rangle\right|^{2}-C_{k}H_{s_{k}}(\xi,\xi)
      \right)\\
      \biggl.&-2C_{k}\Re\left(
     \overline{\langle\partial r_{k-1},\xi\rangle}\langle\partial s_{k},\xi\rangle
      \right)
    \biggr](q).\notag
  \end{align}
  By induction hypothesis \eqref{E:ihypotheses1} we have good control over the first term in 
  \eqref{E:istepint1}:
  \begin{align*}
    e^{-C_{k}s_{k}(q)}H_{r_{k-1}}(\xi,\xi)(q)\geq
    &-\epsilon\left(
    |r_{k}(q)||\xi|^{2}+\frac{1}{|r_{k}(q)|}\left|\langle
    e^{-C_{k}s_{k}(q)}\partial r_{k-1}(q),\xi\rangle
    \right|^{2}
    \right)\\
    &+e^{-C_{k}s_{k}(q)}\mu_{k-1}H_{\rho}(\xi,\xi)(\pi(q)).
  \end{align*}
  Note that
  \begin{align*}
    \left|
    \langle
    e^{-C_{k}s_{k}(q)}\partial r_{k-1}(q),\xi\rangle
    \right|^{2}
    \leq 
    2\left|
    \langle\partial r_{k}(q),\xi\rangle
    \right|^{2}
    +r_{k}^{2}(q)C_{k}^{2}\left|
    \langle\partial s_{k}(q),\xi\rangle
    \right|^{2}.
  \end{align*}
  Moreover, part (iii) of Lemma \ref{L:cutoff} implies that
  \begin{align}\label{E:istepint2}
    C_{k}^{2}\left|
    \langle\partial s_{k}(q),\xi\rangle
    \right|^{2}
    \leq 
    2\epsilon|\xi|^{2}+\mathcal{O}\left(\left|
    \langle\partial r_{k}(\pi(q)),\xi\rangle
    \right|^{2}\right)
    \leq 3\epsilon|\xi|^{2}+\mathcal{O}\left(\left|
    \langle\partial r_{k}(q),\xi\rangle
    \right|^{2}\right)
  \end{align}  
  after possibly shrinking $U_{k-1}$ (in normal direction only). Thus
  we have
  \begin{align*}
     e^{-C_{k}s_{k}(q)}H_{r_{k-1}}(\xi,\xi)(q)\gtrsim
      -\epsilon\left(
    |r_{k}(q)||\xi|^{2}+\frac{1}{|r_{k}(q)|}\left|\langle
    \partial r_{k}(q),\xi\rangle
    \right|^{2}
    \right)+\mu H_{\rho}(\xi,\xi)(\pi(q)).
 \end{align*}
 For some positive constant $\mu\leq\mu_{k-1}$. Thus the first term on the right hand side of  
 \eqref{E:istepint1} is taken care of.
 Now suppose $q\in\Omega\cap U_{k-1}\cap V_{k}$ is such that $\pi(q)\in b\Omega\cap V_{j,k}$ for 
 some $j\in J_{k}$. To be able to deal with the term $H_{s_{k}}(\xi,\xi)(q)$ in \eqref{E:istepint1}, we shall 
 write $\xi=S+W+M$, where
 \begin{align*}
   S\in\Span\left(\{S_{j,k}^{\beta}(\pi(q))\}_{\beta=1}^{k} \right),\;\;
   W\in\Span\left(\{W_{j,k}^{\alpha}(\pi(q))\}_{\alpha=1}^{n-1-k}\right),\;\;\text{and}\;\;
   M\in\Span\left(N(\pi(q))\right).
 \end{align*}
 Then the (sc)-(lc) inequality gives
 \begin{align*}
   C_{k}H_{s_{k}}(\xi,\xi)(q)&\geq C_{k}H_{s_{k}}(W,W)(q)-\frac{\epsilon}{\kappa_{k}}|W|^{2}+
   \mathcal{O}\left(|S|^{2}+|M|^{2}\right)\\
   &\geq -2\frac{\epsilon}{\kappa_{k}}|W|^{2}+\mathcal{O}\left(|S|^{2}+|M|^{2}\right),
 \end{align*}
 where the last step holds since $s_{k}$ satisfies part (iv) of Lemma \ref{L:cutoff}.
 The last inequality together with \eqref{E:istepint2} lets us estimate the second term
 in \eqref{E:istepint1} as follows
 \begin{align*}
      e^{-C_{k}s_{k}(q)}r_{k-1}(q)&\left(
      C_{k}^{2}\left|\langle\partial s_{k},\xi\rangle\right|^{2}-C_{k}H_{s_{k}}(\xi,\xi)
      \right)(q)\\
      &\gtrsim
      -\epsilon
      \left(
    |r_{k}||\xi|^{2}+\frac{1}{|r_{k}|}\left|\langle
    \partial r_{k},\xi\rangle
    \right|^{2}
    \right)(q)+\mathcal{O}(r_{k}(q))|S|^{2}.
  \end{align*}
  For the third term in \eqref{E:istepint1} we use \eqref{E:istepint2} again and obtain
  \begin{align*}
    -2C_{k}e^{-C_{k}s_{k}(q)}\Re\left(
    \overline{\langle\partial r_{k-1},\xi\rangle}\langle \partial s_{k},\xi\rangle
    \right)(q)
    \gtrsim
    -\epsilon\left(|r_{k}(q)||\xi|^{2}+\frac{1}{|r_{k}(q)|}\left|\langle\partial r_{k-1}(q),\xi\rangle
    \right|^{2}\right).
  \end{align*}
  Collecting all these estimates, using $\frac{1}{\kappa_{k}}|W|^{2}\leq|\xi|^{2}$, we now have
  for $q\in\Omega\cap U_{k-1}\cap V_{k}$ with $\pi(q)\in b\Omega\cap  U_{k-1}\cap V_{k}$ and for some  
  $\mu>0$
  \begin{align*}
    H_{r_{k}}(\xi,\xi)(q)\gtrsim
    &-\epsilon
    \left(|r_{k}(q)||\xi|^{2}+\frac{1}{|r_{k}(q)|}\left|\langle\partial r_{k-1}(q),\xi\rangle
    \right|^{2}\right)\\
    &\hspace{4cm}+\mu H_{\rho}(\xi,\xi)(\pi(q))+\mathcal{O}(r_{k}(q))|S|^{2}\\
    \gtrsim&-\epsilon
    \left(|r_{k}(q)||\xi|^{2}+\frac{1}{|r_{k}(q)|}\left|\langle\partial r_{k-1}(q),\xi\rangle
    \right|^{2}\right)
    +\frac{\mu}{2}H_{\rho}(\xi,\xi)(\pi(q)).
  \end{align*}
  Here, the last estimate holds, after possibly shrinking $U_{k-1}\cap V_{k}$ (in normal direction only), 
  since $\rho$ is plurisubharmonic on $b\Omega$.
  
  \medskip
    
  We still need to show that \eqref{E:claimistep} is true in some neighborhood of 
  $\Sigma_{k}\setminus U_{k-1}$. Let $U\subset V_{k}$ be a neighborhood of
  $\Sigma_{k}\setminus U_{k-1}$, $q\in \Omega\cap U$ with $\pi(q)\in b\Omega\cap V_{j,k}$ for
  some $j\in J_{k}$
  and $\xi\in\mathbb{C}^{n}$. Writing $\vartheta_{k}=\vartheta_{k-1}+C_{k}s_{k}$, we get
  \begin{align*}
    H_{r_{k}}(\xi,\xi)(q)
    =e^{-\vartheta_{k}(q)}
    \Bigl[
    H_{\rho}(\xi,\xi)-&2\Re\left(
    \overline{\langle\partial\rho,\xi\rangle}
    \langle\partial\vartheta_{k},\xi\rangle
    \right)\Bigl.\\
    &\Bigr.+\rho\left(
    \left|\langle\partial\vartheta_{k},\xi\rangle \right|^{2}
    -H_{\vartheta_{k-1}}(\xi,\xi)-C_{k}H_{s_{k}}(\xi,\xi)
    \right)
    \Bigr](q)\\
    &=\I+\II+\III+\IV+\V.
  \end{align*}
  We write again $\xi=S+W+M$.
  By construction of $s_{k}$ and by the induction hypotheses \eqref{E:ihypothesis2} and 
  \eqref{E:ihypothesis3} on 
  $\vartheta_{k-1}$, we can do estimates similar to the ones below \eqref{E:istepint1} to obtain
  \begin{align*}
    \II+\III+\IV\gtrsim
    -\epsilon\left(
    |r_{k}(q)||\xi|^{2}+\frac{1}{|r_{k}(q)|}\left|
    \langle\partial r_{k}(q),\xi\rangle
    \right|^{2}
    \right)+\mathcal{O}(r_{k}(q))|S|^{2}\;\;\text{for}\;\;q\in U.
  \end{align*}
  So we are left with the terms $\I$ and $\V$.   
  Let us first consider the term $\I$. By Taylor's Theorem, we have
  \begin{align*}
    e^{-\vartheta_{k}(q)}H_{\rho}(\xi,\xi)(q)
    =
    e^{-\vartheta_{k}(q)}&\Bigl(H_{\rho}(\xi,\xi)(\pi(q))
    -d_{b\Omega}(q)\Re\bigl[ (N
    H_{\rho})(\xi,\xi)(\pi(q))\bigr]\Bigr)\\
    &+\mathcal{O}\left(r_{k}^{2}(q)\right)|\xi|^{2}.
  \end{align*}
  By the (sc)-(lc) inequality we have
  \begin{align*}
    \Re \left[(NH_{\rho})(\xi,\xi)(\pi(q))\right]
    \leq &\Re\left[(NH_{\rho})(W,W)(\pi(q))\right]\\
    &+\frac{\epsilon}{c_{1}\kappa_{k}}|W|^{2}+
    \mathcal{O}\left(\frac{c_{1}\kappa_{k}}{\epsilon}\right)(|S|^{2}+|M|^{2}),
  \end{align*}
  where $c_{1}>0$ is such that $d_{b\Omega}(q)\leq c_{1}|\rho(q)|$.
  Therefore we obtain for some $\mu>0$
  \begin{align*}
    e^{-\vartheta_{k}(q)}H_{\rho}(\xi,\xi)(q)
    \geq
    &-d_{b\Omega}(q)e^{-\vartheta_{k}(q)}\Re\left[(NH_{\rho})(W,W)(\pi(q))\right]
    +r_{k}(q)\frac{\epsilon}{\kappa_{k}}|W|^{2}\\
    &+\mu H_{\rho}(\xi,\xi)(\pi(q))+\mathcal{O}\left(r_{k}(q)\right)(|S|^{2}+|M|^{2})
    +\mathcal{O}\left(r_{k}^{2}(q)\right)|\xi|^{2}.
  \end{align*}
  To estimate term $\V$ we use (sc)-(lc) inequality again:
  \begin{align*}
    -r_{k}(q)C_{k}H_{s_{k}}(\xi,\xi)(q)
    \geq -r_{k}(q)\left(C_{k}H_{s_{k}}(W,W)(q)-\frac{\epsilon}{\kappa_{k}}|W|^{2}\right)+
    \mathcal{O}\left(r_{k}(q)\right)(|S|^{2}+|M|^{2}).
  \end{align*}
  After possibly shrinking $U$, we get for some $\mu>0$
  \begin{align*}
    \I+\V\geq
    &
    -d_{b\Omega}(q)e^{-\vartheta_{k}(q)}\Re\left[(NH_{\rho})(W,W)(\pi(q))\right]
    +r_{k}(q)\left(-C_{k}H_{s_{k}}(W,W)(q)+\frac{2\epsilon}{\kappa_{k}}|W|^{2}
    \right)\\
    &+\mu H_{\rho}(\xi,\xi)(\pi(q))+\mathcal{O}\left(r_{k}(q)\right)|M|^{2}
    +\mathcal{O}\left(r_{k}^{2}(q)\right)|\xi|^{2}.
  \end{align*}
  By our choice of $s_{k}$ and $C_{k}$, it follows that for all $q\in\Omega\cap U$ with
  $\pi(q)\in b\Omega\cap U$ we have
  \begin{align*}
    -d_{b\Omega}e^{-\vartheta_{k}(q)}\Re\left[ (NH_{\rho})(W,W)(\pi(q))\right]
    -r_{k}(q)C_{k}H_{s_{k}}(W,W)(q)\geq\frac{\epsilon}{\kappa_{k}}r_{k}(q)|W|^{2}.
  \end{align*}
  Putting our estimates for the terms $\I$--$\V$ together and letting
   $U_{k}$ be the union of $U_{k-1}$ and 
  $U$, we obtain: for all $q\in\Omega\cap U_{k}$ with $\pi(q)\in b\Omega\cap U_{k}$, the 
  function $r_{k}$ satisfies
  \begin{align*}
    H_{r_{k}}(\xi,\xi)(q)\gtrsim -\epsilon\left(
    |r_{k}(q)||\xi|^{2}+\frac{1}{|r_{k}(q)|}\left| 
    \langle\partial r_{k}(q),\xi\rangle
    \right|^{2}
    \right)+\mu_{k}H_{\rho}(\xi,\xi)(\pi(q))
  \end{align*}
  for some $\mu_{k}>0$.
  Since the constants in $\gtrsim$ do
   not depend on $\epsilon$ or on any other parameters which come 
  up in the construction of $r_{k}$, \eqref{E:claimistep} follows. Moreover, by construction, $\vartheta_{k}$  
  satisfies \eqref{E:ihypothesis2} and \eqref{E:ihypothesis3}. 
   
   Note that $\Sigma_{n-1}\setminus U_{n-2}$ is 
   a closed subset of the set of strictly pseudoconvex boundary points. Thus for
   any smooth defining function $r$ there is some
   neighborhood $U$ of $\Sigma_{n-1}\setminus U_{n-2}$ such that
   \begin{align*}
     H_{r}(\xi,\xi)\gtrsim |\xi|^{2}
     +\mathcal{O}(|\langle\partial r(q),\xi\rangle|^{2})\;\;\text{for all}\;\;\xi\in\mathbb{C}^{n}
   \end{align*}
   holds on $U\cap\Omega$.
  This concludes the proof of \eqref{E:Main1}.
  
  \medskip
  
  The proof of \eqref{E:Main2} is essentially the same as the one of \eqref{E:Main1} except that a few 
  signs change. That is, the basic estimate \eqref{E:BasicTaylorH} for $q\in\overline{\Omega}^{c}\cap U$
  becomes
  \begin{align*}
    H_{\rho}(W,W)(q)=2d_{b\Omega}(q)NH_{\rho}(W,W)(\pi(q))+\mathcal{O}\left(d_{b\Omega}^{2}(q)
    \right)|W|^{2}
  \end{align*}
   for any vector $W\in\mathbb{C}^{n}$ which is a weak complex tangential direction at $\pi(q)$.
   So an obstruction for \eqref{E:Main2} to hold at $q\in\overline{\Omega}^{c}\cap U$ 
   occurs when $NH_{\rho}(W,W)$ is negative at $\pi(q)$ -- 
   note that this happens exactly when we have no problem with \eqref{E:Main1}. 
   Since the obstruction terms
   to \eqref{E:Main1} and \eqref{E:Main2} only differ by a sign, one would expect that the necessary 
   modifications of 
   $\rho$ also just differ by a sign. In fact, let $\vartheta_{n-2}$ be as in the proof of \eqref{E:Main1} --  
   that is, $r_{1}=\rho e^{-\vartheta_{n-2}}$ satisfies \eqref{E:Main1} for a given $\epsilon>0$. Then
   $r_{2}=\rho e^{\vartheta_{n-2}}$ satisfies \eqref{E:Main2} for the same $\epsilon$.

   \section{Proof of Corollary \ref{C:DF}}\label{S:DF}
     We shall now prove Corollary \ref{C:DF}. We begin with part (i) by 
     showing  first that for 
     any $\eta\in(0,1)$ there exist a $\delta>0$, a smooth defining function $r$ of $\Omega$ and a 
     neighborhood $U$ of $b\Omega$  such that
     $h=-(-re^{-\delta|z|^{2}})^{\eta}$ is strictly plurisubharmonic on $\Omega\cap U$.
     
     Let $\eta\in(0,1)$ be fixed, and $r$ be a smooth defining function of $\Omega$. For notational ease 
     we write $\phi(z)=\delta|z|^{2}$ for $\delta>0$. Here, $r$ and $\delta$ are fixed and to be chosen later.
     Let us compute the complex Hessian of $h$ on $\Omega\cap U$:
     \begin{align*}
       H_{h}(\xi,\xi)=&\eta(-r)^{\eta-2}e^{-\phi\eta}
       \Bigl[
       (1-\eta)\Bigr.\left|\langle\partial r,\xi\rangle\right|^{2}-rH_{r}(\xi,\xi)\\
       &+2r\eta \Re\left(\langle\partial r,\xi\rangle\langle\overline{\partial\phi,\xi}\rangle\right)
       \Bigl.
       -r^{2}\eta\left|\langle\partial\phi,\xi\rangle\right|^{2}
      +r^{2}H_{\phi}(\xi,\xi)\Bigr].
     \end{align*}
     An application of the (sc)-(lc) inequality gives
     \begin{align*}
       2r\eta \Re\left(\langle\partial r,\xi\rangle\langle\overline{\partial\phi,\xi}\rangle\right)
       \geq -\frac{1-\eta}{2}\left|\langle\partial r,\xi\rangle\right|^{2}
       -\frac{2r^{2}\eta^{2}}{1-\eta}\left|\langle\partial\phi,\xi\rangle\right|^{2}.
     \end{align*}
     Therefore, we obtain for the complex Hessian of $h$ on $\Omega$ the following:
     \begin{align*}
       H_{h}(\xi,\xi)
       \geq
       \eta(-r)^{\eta-2}e^{-\phi\eta}
       \Biggl[\frac{1-\eta}{2}|\langle\partial r,\xi\rangle|^{2}
       \Biggr.&-rH_{r}(\xi,\xi)\\
       &+r^{2}
       \left\{
       -\frac{\eta(1+\eta)}{1-\eta}|\langle\partial\phi,\xi\rangle|^{2}+H_{\phi}(\xi,\xi)
       \right\}
       \Biggl.\Biggr].
     \end{align*}
     Set  $\delta=\frac{1-\eta}{2\eta(1+\eta)D}$, where $D=\max_{z\in\overline{\Omega}}|z|^{2}$. Then we 
     get
     \begin{align*}
       H_{\phi}(\xi,\xi)-\frac{\eta(1+\eta)}{1-\eta}\left|\langle\partial\phi,\xi\rangle\right|^{2}
       =
       \delta\left(H_{|z|^{2}}(\xi,\xi)-\frac{\eta(1+\eta)}{1-\eta}\delta\left|\langle\overline{z},
       \xi\rangle\right|^{2}
       \right)
       \geq
       \frac{\delta}{2}|\xi|^{2}.
      \end{align*}
       This implies that
     \begin{align}\label{E:generalDFest}
       H_{h}(\xi,\xi)\geq \eta(-r)^{\eta-2}e^{-\phi\eta}
       \left[
       \frac{1-\eta}{2}|\langle \partial r,\xi\rangle|^{2}
       -rH_{r}(\xi,\xi)+\frac{\delta}{2}r^{2}|\xi|^{2}
       \right]
     \end{align}
     holds on $\Omega$. 
     
     Set $\epsilon=\min\{\frac{1-\eta}{4},\frac{1-\eta}{8\eta(1+\eta)D}\}$.
     By \eqref{E:Main1} 
     there exist a neighborhood $U$ of $b\Omega$ and a smooth defining function $r_{1}$ 
     of $\Omega$ such that
     \begin{align*}
     H_{r_{1}}(\xi,\xi)(q)
     \geq
     -\epsilon\left(|r_{1}(q)||\xi|^{2}+\frac{1}{|r_{1}(q)|}|\langle\partial r_{1}(q),\xi\rangle|^{2}\right)
     \end{align*}
     holds for all $q\in\Omega\cap U$, $\xi\in\mathbb{C}^{n}$. 
     Setting $r=r_{1}$ and using \eqref{E:generalDFest}, we obtain
     \begin{align*}
       H_{h}(\xi,\xi)(q)\geq
       \eta(-h(q))\cdot\epsilon|\xi|^{2}\;\;\text{for}\;\;q\in\Omega\cap U,\;\xi\in\mathbb{C}^{n}.
     \end{align*}
     It follows by standard arguments that there exists a defining function $\widetilde{r}_{1}$ such that 
     $-(-\widetilde{r}_{1})^{\eta}$ is strictly 
     plurisubharmonic on $\Omega$; for details see pg.\ 133 in \cite{DF1}. 
      This proves part (i) of Corollary \ref{C:DF}.
     
     \medskip
     A proof similar to the one of part (i), using \eqref{E:Main2}, 
     shows that  for each $\eta>1$
     there exists a smooth
     defining function $\tilde{r}_{2}$, a neighborhood $U$ of $b\Omega$ and $\delta>0$ such that
     $(r_{2}e^{\delta|z|^{2}})^{\eta}$ is strictly plurisubharmonic on $\overline{\Omega}^{c}\cap U$.

\end{document}